\documentclass[12pt]{article}
\usepackage{amsmath,amsthm,mathrsfs}
\usepackage{amsfonts}
\usepackage{latexsym}
\usepackage{cite}

\usepackage{t1enc}
\usepackage[latin1]{inputenc}
\usepackage[english]{babel}
\usepackage[dvips]{graphicx}
\usepackage{graphicx}
\usepackage[natural]{xcolor}
\newtheorem{theorem}{Theorem}
\newtheorem{lemma}[theorem]{Lemma}
\newtheorem{proposition}[theorem]{Proposition}
\newtheorem{conjecture}[theorem]{Conjecture}

\newtheorem{problem}[theorem]{Problem}

\newcommand{\diam}{{\rm diam}}

\begin{document}

\title{On distance-balanced generalized Petersen graphs}

\author{Gang Ma$^{a}$ \and Jianfeng Wang$^{a,}$\footnote{Corresponding author} \and Sandi Klav\v{z}ar$^{b,c,d}$}

\date{}

\maketitle
\vspace{-0.8 cm}
\begin{center}
$^a$ School of Mathematics and Statistics, Shandong University of Technology\\
 Zibo, China\\
{\tt math$\_$magang@163.com\\
jfwang@sdut.edu.cn}\\
\medskip

$^b$ Faculty of Mathematics and Physics, University of Ljubljana, Slovenia\\
{\tt sandi.klavzar@fmf.uni-lj.si}\\
\medskip

$^c$ Faculty of Natural Sciences and Mathematics, University of Maribor, Slovenia\\
\medskip

$^d$ Institute of Mathematics, Physics and Mechanics, Ljubljana, Slovenia\\
\end{center}

\begin{abstract}
A connected graph $G$ of diameter ${\rm diam}(G) \ge \ell$  is $\ell$-distance-balanced if $|W_{xy}|=|W_{yx}|$ for every $x,y\in V(G)$ with $d_{G}(x,y)=\ell$, where $W_{xy}$ is the set of vertices of $G$ that are closer to $x$ than to $y$. We prove that the generalized Petersen graph $GP(n,k)$ is ${\rm diam}(GP(n,k))$-distance-balanced provided that $n$ is large enough relative to $k$. This partially solves a conjecture posed by  Miklavi\v{c} and \v{S}parl \cite{Miklavic:2018}. We also determine  ${\rm diam}(GP(n,k))$ when $n$ is large enough relative to $k$.
\end{abstract}

\noindent {\bf Key words:} generalized Petersen graph; distance-balanced graph; $\ell$-distance-balanced graph; diameter

\medskip\noindent
{\bf AMS Subj.\ Class:} 05C12

\section{Introduction}
\label{S:intro}

If $G = (V(G), E(G))$ is a connected graph and $x, y\in V(G)$, then the {\it distance} $d_{G}(x, y)$ between $x$ and $y$ is the number of edges on a shortest $x,y$-path. The diameter $\diam(G)$ of $G$ is the maximum distance between its vertices. The set $W_{xy}$ contains the vertices that are closer to $x$ than to $y$, that is,
$$W_{xy}=\{w\in V(G):\ d_{G}(w,x) < d_{G}(w,y)\}\,.$$
Vertices $x$ and $y$ are {\em balanced} if $|W_{xy}| = |W_{yx}|$.  For an integer $\ell \in [\diam(G)] = \{1,2,\ldots, \diam(G)\}$ we say that $G$ is $\ell$-{\em distance-balanced} if each pair of vertices $x,y\in V(G)$ with $d_{G}(x,y) = \ell$ is balanced. $G$ is said to be {\em highly distance-balanced} if it is $\ell$-distance-balanced for every $\ell\in [\diam(G)]$. $1$-distance-balanced graphs are simply called {\em distance-balanced} graphs.

Distance-balanced graphs were first considered by Handa~\cite{Handa:1999} back in 1999, while the term ``distance-balanced'' was proposed a decade later by Jerebic et al.\ in~\cite{Jerebic:2008}. The latter paper was the trigger for intensive research of distance-balanced graphs, see~\cite{Abiad:2016, Balakrishnan:2014, Balakrishnan:2009, Cabello:2011, cavaleri-2020, fernardes-2022, Ilic:2010, Kutnar:2006, Kutnar:2009, Kutnar:2014, Miklavic:2012, YangR:2009}.
The study of distance-balanced graphs is interesting from various purely graph-theoretic aspects where one focuses on particular properties of such graphs such as symmetry, connectivity or complexity aspects of algorithms related to such graphs.
Moreover, distance-balanced graphs have motivated the introduction of the hitherto much-researched Mostar index~\cite{ali-2021, doslic-2018} and distance-unbalancedness of graphs~\cite{kr-2021, miklavic-2021, xu-2022}. In this context, distance-balanced graphs are the graphs with the Mostar index equal to 0.

In his dissertation~\cite{Frelih:2014}, Frelih generalized distance-balanced graphs to $\ell$-distance balanced graphs. The special case of $\ell=2$ has been studied in detail in~\cite{Frelih:2018}. Among other results it was demonstrated that there exist $2$-distance-balanced graphs that are not $1$-distance-balanced. $2$-distance-balanced graphs that are not $2$-connected were characterized as well as $2$-distance-balanced Cartesian and lexicographic products. In this direction, $\ell$-distance-balanced corona products and lexicographic products were investigated in~\cite{Jerebic:2021}. In~\cite{Miklavic:2018}, Miklavi\v{c} and \v{S}parl obtained some general results on $\ell$-distance balanced graphs. They studied graphs of diameter at most $3$ and investigated $\ell$-distance-balancedness of cubic graphs, in particular of generalized Petersen graphs.
Although generalized Petersen graphs are a family of cubic graphs but it is difficult to determine whether they are $\ell$-distance-balanced or not for some $\ell$.
And that is what has stimulated the main interest in this article. Before we explain this in more detail, let us define these graphs.

If $n\ge 3$ and $1\le k<n/2$, then the {\em generalized Petersen graph} $GP(n,k)$ is defined by
\begin{align*}
V(GP(n,k)) & = \{u_i:\ i\in \mathbb{Z}_n\} \cup\{v_i:\ i\in \mathbb{Z}_n\}, \\
E(GP(n,k)) & = \{u_iu_{i+1}:\ i\in \mathbb{Z}_n\} \cup\{v_iv_{i+k}:\ i\in \mathbb{Z}_n\} \cup \{u_iv_i:\ i\in \mathbb{Z}_n\}.
\end{align*}
$GP(6,2)$ is shown in Figure~\ref{F:GP(6,2)}.
\begin{figure}[htbp]
\begin{center}
\includegraphics[scale=0.9]{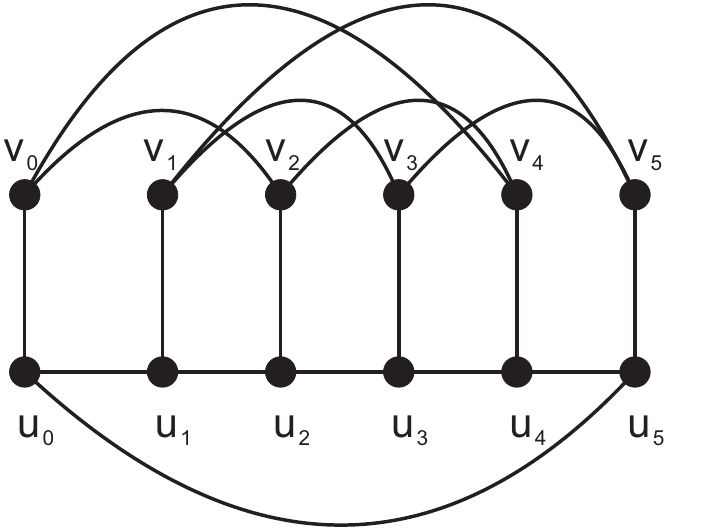}\\
\caption{\small The generalized Petersen graph $GP(6,2)$. The cycle $u_0u_1u_2u_3u_4u_5$ contains six outer edges. $v_0v_2$, $v_1v_3$, $v_2v_4$, $v_3v_5$, $v_4v_0$ and $v_5v_1$ are six inner edges.
Finally $u_0v_0$, $u_1v_1$, $u_2v_2$, $u_3v_3$, $u_4v_4$ and $u_5v_5$ are six spokes.}\label{F:GP(6,2)}
\end{center}
\end{figure}

Now, we recall the following conjecture and result, where the conjecture was supported by an extensive computer search.

\begin{conjecture} {\rm \cite[Conjecture 5.2]{Miklavic:2018}}
\label{C:GP-distance}
If $n\ge 3$, $2\le k< n/2$, and there exists $j\in\mathbb{Z}_n$ such that $d(u_0,v_j) = \diam(GP(n,k))$, then either $n=4m$ and $k=2m-1$ for some $m\ge 3$, or $(n,k)\in \{(5,2), (7,2), (7,3)\}$.
\end{conjecture}

\begin{proposition} {\rm \cite[Proposition 5.3]{Miklavic:2018}}
\label{prop:stefko}
If $n\ge 3$, $2\le k< n/2$, and if Conjecture~\ref{C:GP-distance} holds,  then $GP(n,k)$ is $\diam(GP(n,k))$-distance-balanced.
\end{proposition}

The main result of this paper reads as follows.

\begin{theorem}\label{T:GP-DDB}
If $n$ and $k$ are integers, where $3\le k< n/2$ and
\begin{equation*}
n\ge\left\{\begin{array}{ll}
8; & k=3, \\
10; & k=4, \\
\frac{k(k+1)}{2}; & k \ \text{is odd and}\ k\ge 5,\\
\frac{k^2}{2}; & k \ \text{is even and}\ k\ge 6,
\end{array}\right.
\end{equation*}
then $GP(n,k)$ is $\diam(GP(n,k))$-distance-balanced.
\end{theorem}

Theorem~\ref{T:GP-DDB} is proved in Section~\ref{S:GP-DDB}. In view of Proposition~\ref{prop:stefko},
to prove Theorem~\ref{T:GP-DDB} it suffices to verify that Conjecture~\ref{C:GP-distance} holds true for the cases as listed in the theorem. The difficulty in proving Conjecture~\ref{C:GP-distance} in general lies in the fact that the distance function on generalized Petersen graphs is very difficult to manage and depends heavily on $n$ and $k$. In particular, as pointed out by Miklavi\v{c} and \v{S}parl in~\cite[p.~150]{Miklavic:2018}, the diameter of $GP(n,k)$ is not known in general. In Section~\ref{S:GP-diameter}  we then determine $\diam(GP(n,k))$ for the corresponding values of $n$ and $k$. The rather complicated result indicates that it is indeed difficult to control the diameter of generalized Petersen graphs. Finally, in Section~\ref{S:conluding}, we list some problems which are worth studying in the future.

\section{Proof of Theorem~\ref{T:GP-DDB}}
\label{S:GP-DDB}

Consider the generalized Petersen graph $GP(n,k)$. The edges of the form $u_iu_{i+1}$ are {\em outer} edges, the edges of the form $v_iv_{i+k}$ are {\em inner} edges, and edges of the form $u_iv_i$ are {\em spokes}. To simplify the notation, set $D = \diam(GP(n,k))$ throughout this section. We will also omit the subscript in $d_{GP(n,k))}(x,y)$ as the graph $GP(n,k)$ is clear from the context.

As already stated at the end of the previous section, in order to prove Theorem~\ref{T:GP-DDB}, it suffices to prove that if $n$ and $k$ are integers, where $3\le k< n/2$ and
\begin{equation*}
n\ge\left\{\begin{array}{ll}
8; & k=3, \\
10; & k=4, \\
\frac{k(k+1)}{2}; & k \ \text{is odd and}\ k\ge 5,\\
\frac{k^2}{2}; & k \ \text{is even and}\ k\ge 6,
\end{array}\right.
\end{equation*}
then for any $j\in\mathbb{Z}_n$ we have $d(u_0,v_j) < D$.

By the symmetry of $GP(n,k)$ it suffices to consider $d(u_0, v_j)$, where $0\le j\le n/2$. Our aim is to find an index $j^*$, where $0\le j^*\le n/2$, such that
$$d(u_0,v_{j^*}) = \max\{d(u_0,v_j):\ 0\le j\le n/2\}\,,$$
and prove that $d(u_0,v_{j^*})<D$.

Let $j$ be an integer such that $1\le j\le n/2$. Suppose $j=m_0k+j_0$ and $n-j=m_1k+j_1$, where $0\le j_0,j_1<k$.
Four types of $u_0,v_j$-path are defined in the following.
\begin{align*}
P_1 & = u_0u_1u_2\cdots u_{j_0}v_{j_0}v_{k+j_0}v_{2k+j_0}\cdots v_{m_0k+j_0}, \\
P_2 & = u_0u_{-1}u_{-2}\cdots u_{-(k-j_0)}v_{-(k-j_0)}v_{j_0}v_{k+j_0}\cdots v_{m_0k+j_0}, \\
P_3 & = u_0u_{-1}u_{-2}\cdots u_{-{j_1}}v_{-j_1}v_{-(k+j_1)}v_{-(2k+j_1)}\cdots v_{-(m_1k+j_1)}, \\
P_4 & = u_0u_1u_2\cdots u_{k-j_1}v_{k-j_1}v_{-j_1}v_{-k-j_1}\cdots v_{-m_1k-j_1}.
\end{align*}
Note that $u_{-i}=u_{n-i}$, so $v_{-m_1k-j_1}=v_{n-m_1k-j_1}=v_j=v_{m_0k+j_0}$.
Also note that all $P_1,P_2,P_3,P_4$ have only one spoke.
The length of $P_1$ is $j_0+m_0+1$, the length of $P_2$ is $(k-j_0)+m_0+2$, the length of $P_3$ is $j_1+m_1+1$,
and the length of $P_4$ is $(k-j_1)+m_1+2$.

In $GP(6,2)$, the $u_0,v_3$-path of type $P_1$ is $u_0u_1v_1v_3$.
The $u_0,u_3$-path of type $P_2$ is $u_0u_{-1}v_{-1}v_1v_3=u_0u_{5}v_{5}v_1v_3$.
The $u_0,u_3$-path of type $P_3$ is $u_0u_{-1}v_{-1}v_{-3}=u_0u_{5}v_{5}v_{3}$.
The $u_0,u_3$-path of type $P_4$ is $u_0u_1v_1v_{-1}v_{-3}=u_0u_1v_1v_{5}v_{3}$.

We first prove the following lemma about the $u_0, v_j$-path of $GP(n,k)$.
\begin{lemma}
Suppose that two integers $k,n$ and four paths $P_1,P_2,P_3,P_4$ are the same as above.
In $GP(n,k)$, for any integer $j$ where $1\le j\le n/2$,
a $u_0, v_j$-path of minimum length contains only one spoke and belongs to one of the four types $\{P_1,P_2,P_3,P_4\}$.
\end{lemma}
\begin{proof}
For the convenience of computing the distance of the path in $GP(n,k)$, we divide the direction of the path
into $positive$ and $negative$ direction, and use different vertex subscripts marking method according to the direction of the path.

Suppose that there is a path from $u_{i_1}$ via outer edges.
If the path from $u_{i_1}$ is of $positive$ direction, then the path is denoted by $u_{i_1}u_{i_1+1}u_{i_1+2}\cdots$.
If the path from $u_{i_1}$ is of $negative$ direction, then the path is denoted by $u_{i_1}u_{i_1-1}u_{i_1-2}\cdots$.
Using the above vertex subscripts marking method, the distance of a $u_{i_1},u_{i_2}$-path (via both directions)
via outer edges is $|i_2-i_1|$.

Similarly, suppose that there is a path from $v_{i_1}$ via inner edges.
If the path from $v_{i_1}$ is of $positive$ direction, then the path is denoted by $v_{i_1}v_{i_1+k}v_{i_1+2k}\cdots$.
If the path from $v_{i_1}$ is of $negative$ direction, then the path is denoted by $v_{i_1}v_{i_1-k}v_{i_1-2k}\cdots$.
Using the above vertex subscripts marking method, the distance of a $v_{i_1},v_{i_2}$-path (via both directions)
via inner edges is $|\frac{i_2-i_1}{k}|$.

Whenever considering the path that connects $u_{i_1}$ to $u_{i_2}$ via outer edges,
it is always negative if $i_2 < i_1$, and positive otherwise. Same for the path that
connects $v_{i_2}$ to $v_{i_3}$ via inner edges.

{\bf Claim 1.} A $u_0, v_j$-path of minimum length contains only one spoke.

Let $J^{}=j+rn$ where $r$ is an integer. Note that $v_{J}=v_j$.

Note that a $u_0, v_J$-path cannot contain even number of spokes.
Let $P^{(1)}$ be a $u_0, v_J$-path containing $3$ spokes.
Suppose that $P^{(1)}$ connects $u_0$ and $u_{i_1}$ via outer edges, then spoke $u_{i_1}v_{i_1}$,
then connects $v_{i_1}$ and $v_{i_2}$ via inner edges, then spoke $v_{i_2}u_{i_2}$,
then connects $u_{i_2}$ and $u_{i_3}$ via outer edges, then spoke $u_{i_3}v_{i_3}$, and then connects
$v_{i_3}$ and $v_J$ via inner edges.

Let $P^{(2)}$ be the $u_0, v_J$-path that connects $u_0$ and $u_{i_1+i_3-i_2}$ via outer edges,
then spoke $u_{i_1+i_3-i_2}v_{i_1+i_3-i_2}$, and then connects
$v_{i_1+i_3-i_2}$ and $v_J$ via inner edges.

Let $LEN(P)$ be the length of path $P$. Then
\begin{align*}
LEN(P^{(1)})&=(|i_1|+1+|\frac{i_2-i_1}{k}|+1+|i_3-i_2|+1+|\frac{J-i_3}{k}|),\\
LEN(P^{(2)})&=(|i_1+i_3-i_2|+1+|\frac{J-i_1-i_3+i_2}{k}|).
\end{align*}
Because $|a+b|\le |a|+|b|$ for two integers, $|i_1+i_3-i_2|\le|i_1|+|i_3-i_2|$
and $|\frac{J-i_1-i_3+i_2}{k}|\le |\frac{J-i_3}{k}|+|\frac{i_2-i_1}{k}|$.
We get $LEN(P^{(1)})-LEN(P^{(2)})\ge 2$.
So $P^{(1)}$ is a $u_0, v_j$-path but not of minimum length.

If a $u_0, v_J$-path contains $5$ or more than $5$ spokes, similar transformation like above can give
a new $u_0, v_J$-path which has smaller spokes and smaller length than the original $u_0, v_J$-path.

{\bf Claim 2.} A $u_0, v_j$-path of minimum length belongs to one of the four types $\{P_1,P_2,P_3,P_4\}$.

Let $P^{(3)}$ be a $u_0, v_J$-path with one spoke.
Suppose that $P^{(3)}$ connects $u_0$ and $u_{i_1}$
via outer edges, then spoke $u_{i_1}v_{i_1}$, and then connects $v_{i_1}$ and $v_J$ via inner edges.

Firstly we prove that $P^{(3)}$ is not a minimum $u_0, v_j$-path if $|i_1|\ge k$.
Suppose $|i_1|\ge k$.
Let $i_1=sk+t$ such that $s\ge 1$ and $0\le t<k$ when $i_1\ge k$, and $s\le -1$ and $-k<t\le 0$ when $i_1\le -k$.

Let $P^{(4)}$ be the $u_0, v_J$-path which connects $u_0$ and $u_t$ via outer edges
(with the same direction as the $u_0,u_{i_1}$-path via outer edges in $P^{(3)}$), then spoke $u_tv_t$, and then
connects $v_t$ and $v_J$ via inner edges (with the same direction as the $v_{i_1},v_J$-path via inner edges in $P^{(3)}$).


We discuss the following four cases.

(1) In $P^{(3)}$, the $u_0,u_{i_1}$-path via outer edges is of positive direction
and the $v_{i_1},v_J$-path via inner edges is of positive direction.

Note that
\begin{align*}
LEN(P^{(3)})&=|i_1|+1+|\frac{J-i_1}{k}|=i_1+1+\frac{J-i_1}{k}, \\
LEN(P^{(4)})&=|t|+1+|\frac{J-t}{k}|=t+1+\frac{J-t}{k}.
\end{align*}

$LEN(P^{(3)})-LEN(P^{(4)})=s(k-1)>0$.
So $P^{(3)}$ is a $u_0, v_j$-path but not of minimum length.

(2) In $P^{(3)}$, the $u_0,u_{i_1}$-path via outer edges is of positive direction
and the $v_{i_1},v_J$-path via inner edges is of negative direction.

Note that
\begin{align*}
LEN(P^{(3)})&=|i_1|+1+|\frac{J-i_1}{k}|=i_1+1+\frac{i_1-J}{k}, \\
LEN(P^{(4)})&=|t|+1+|\frac{J-t}{k}|=t+1+\frac{t-J}{k}.
\end{align*}

$LEN(P^{(3)})-LEN(P^{(4)})=s(k+1)>0$.
So $P^{(3)}$ is a $u_0, v_j$-path but not of minimum length.

(3) In $P^{(3)}$, the $u_0,u_{i_1}$-path via outer edges is of negative direction
and the $v_{i_1},v_J$-path via inner edges is of positive direction.

Note that
\begin{align*}
LEN(P^{(3)})&=|i_1|+1+|\frac{J-i_1}{k}|=-i_1+1+\frac{J-i_1}{k}, \\
LEN(P^{(4)})&=|t|+1+|\frac{J-t}{k}|=-t+1+\frac{J-t}{k}.
\end{align*}

$LEN(P^{(3)})-LEN(P^{(4)})=-s(k+1)>0$.
So $P^{(3)}$ is a $u_0, v_j$-path but not of minimum length.

(4) In $P^{(3)}$, the $u_0,u_{i_1}$-path via outer edges is of negative direction
and the $v_{i_1},v_J$-path via inner edges is of negative direction.

Note that
\begin{align*}
LEN(P^{(3)})&=|i_1|+1+|\frac{J-i_1}{k}|=-i_1+1+\frac{i_1-J}{k}, \\
LEN(P^{(4)})&=|t|+1+|\frac{J-t}{k}|=-t+1+\frac{t-J}{k}.
\end{align*}

$LEN(P^{(3)})-LEN(P^{(4)})=-s(k-1)>0$.
So $P^{(3)}$ is a $u_0, v_j$-path but not of minimum length.

Secondly we prove that $P^{(3)}$ is not a minimum $u_0, v_j$-path if $J>n$ or $J<-n$ (that is $r\neq 0,-1$).
Suppose that $|i_1|<k$ and $r\neq 0,-1$.
We discuss the following four cases.

(1) $0\le i_1<k$ and $r\ge 1$.

In this case, the $u_0,u_{i_1}$-path via outer edges is of positive direction and the $v_{i_1},v_{J}$-path
via inner edges is of positive direction. Note that\\
$LEN(P^{(3)})=|i_1|+1+|\frac{j+rn-i_1}{k}|=i_1+1+\frac{j+rn-i_1}{k}$.

(1.1) When $k$ is odd and $j_0\le\frac{k+1}{2}$, or $k$ is even and $j_0\le\frac{k}{2}$.

We compare $LEN(P^{(3)})$ with $LEN(P_1)$.

If $i_1\ge j_0$, $LEN(P^{(3)})-LEN(P_1)=(i_1+1+\frac{j+rn-i_1}{k})-(j_0+1+\frac{j-j_0}{k})=\frac{rn-(i_1-j_0)}{k}+(i_1-j_0)>0$.

If $i_1<j_0$, $1\le j_0-i_1\le\frac{k+1}{2}$ when $k$ is odd (or $1\le j_0-i_0\le\frac{k}{2}$ when $k$ is even).
Recall that $n\ge \frac{k(k+1)}{2}$ when $k$ is odd (or $n\ge \frac{k^2}{2}$ when $k$ is even). Then
\begin{align*}
LEN(P^{(3)})-LEN(P_1)&=\frac{rn+j_0-i_1}{k}-(j_0-i_1) \\
&>\frac{rk(k+1)/2}{k}-\frac{k+1}{2}\\
&=\frac{(r-1)(k+1)}{2}\ge 0
\end{align*}
when $k$ is odd and
\begin{align*}
LEN(P^{(3)})-LEN(P_1)&=\frac{rn+j_0-i_1}{k}-(j_0-i_1)\\
&>\frac{rk^2/2}{k}-\frac{k}{2}\\
&=\frac{(r-1)k}{2}\ge 0
\end{align*}
when $k$ is even.
So $P^{(3)}$ is a $u_0, v_j$-path but not of minimum length.

(1.2) When $k$ is odd and $j_0>\frac{k+1}{2}$, or $k$ is even and $j_0>\frac{k}{2}$.

We compare $LEN(P^{(3)})$ with $LEN(P_2)$.

If $i_1\ge j_0$, $i_1>k-j_0$ and so $i_1+j_0-k>0$. Then
$LEN(P^{(3)})-LEN(P_2)=(i_1+1+\frac{j+rn-i_1}{k})-(k-j_0+2+\frac{j-j_0}{k})=\frac{rn-(i_1-j_0)}{k}+(i_1+j_0-k-1)>0$.

If $i_1<j_0$, $j_0-i_1\ge 1$.
Recall that $n\ge \frac{k(k+1)}{2}$ when $k$ is odd (or $n\ge \frac{k^2}{2}$ when $k$ is even). Then
\begin{align*}
LEN(P^{(3)})-LEN(P_2)&=\frac{rn+j_0-i_1}{k}-(k+1-i_1-j_0) \\
&>\frac{rk(k+1)/2}{k}-\frac{k+1}{2}\\
&=\frac{(r-1)(k+1)}{2}\ge 0
\end{align*}
when $k$ is odd and
\begin{align*}
LEN(P^{(3)})-LEN(P_2)&=\frac{rn+j_0-i_1}{k}-(k+1-i_1-j_0)\\
&>\frac{rk^2/2}{k}-\frac{k}{2}\\
&=\frac{(r-1)k}{2}\ge 0
\end{align*}
when $k$ is even.
So $P^{(3)}$ is a $u_0, v_j$-path but not of minimum length.

(2) $0\le i_1<k$ and $r\le -2$.

In this case, the $u_0,u_{i_1}$-path via outer edges is of positive direction and the $v_{i_1},v_{J}$-path
via inner edges is of negative direction. Note that\\
$LEN(P^{(3)})=|i_1|+1+|\frac{j+rn-i_1}{k}|=i_1+1+\frac{i_1-(j+rn)}{k}$.

(2.1) When $k$ is odd and $j_0\le\frac{k+1}{2}$, or $k$ is even and $j_0\le\frac{k}{2}$.

We compare $LEN(P^{(3)})$ with $LEN(P_1)$.

If $i_1\ge j_0$, $LEN(P^{(3)})-LEN(P_1)=(i_1+1+\frac{i_1-(j+rn)}{k})-(j_0+1+\frac{j-j_0}{k})=\frac{i_1+j_0-rn-2j}{k}+(i_1-j_0)>0$.

If $i_1<j_0$, $1\le j_0-i_1\le\frac{k+1}{2}$ when $k$ is odd (or $1\le j_0-i_0\le\frac{k}{2}$ when $k$ is even).
Recall that $n\ge \frac{k(k+1)}{2}$ when $k$ is odd (or $n\ge \frac{k^2}{2}$ when $k$ is even). Then
\begin{align*}
LEN(P^{(3)})-LEN(P_1)&=\frac{i_1+j_0-rn-2j}{k}-(j_0-i_1) \\
&\ge \frac{i_1+j_0-rn-n}{k}-\frac{k+1}{2}\\
&>\frac{(-r-1)k(k+1)/2}{k}-\frac{k+1}{2}\\
&=\frac{(-r-2)(k+1)}{2}\ge 0
\end{align*}
when $k$ is odd and
\begin{align*}
LEN(P^{(3)})-LEN(P_1)&=\frac{i_1+j_0-rn-2j}{k}-(j_0-i_1) \\
&\ge \frac{i_1+j_0-rn-n}{k}-\frac{k}{2}\\
&>\frac{(-r-1)k^2/2}{k}-\frac{k}{2}\\
&=\frac{(-r-2)k}{2}\ge 0
\end{align*}
when $k$ is even.
So $P^{(3)}$ is a $u_0, v_j$-path but not of minimum length.

(2.2) When $k$ is odd and $j_0>\frac{k+1}{2}$, or $k$ is even and $j_0>\frac{k}{2}$.

We compare $LEN(P^{(3)})$ with $LEN(P_2)$.

If $i_1\ge j_0$, $i_1>k-j_0$ and so $i_1+j_0-k>0$. Then
$LEN(P^{(3)})-LEN(P_2)=(i_1+1+\frac{i_1-(j+rn)}{k})-(k-j_0+2+\frac{j-j_0}{k})=\frac{i_1+j_0-rn-2j}{k}+(i_1+j_0-k-1)>0$.

If $i_1<j_0$, $j_0-i_1\ge 1$.
Recall that $n\ge \frac{k(k+1)}{2}$ when $k$ is odd (or $n\ge \frac{k^2}{2}$ when $k$ is even). Then
\begin{align*}
LEN(P^{(3)})-LEN(P_2)&=\frac{i_1+j_0-rn-2j}{k}-(k+1-i_1-j_0) \\
&>\frac{i_1+j_0-rn-n}{k}-\frac{k+1}{2}\\
&>\frac{(-r-1)k(k+1)/2}{k}-\frac{k+1}{2}\\
&=\frac{(-r-2)(k+1)}{2}\ge 0
\end{align*}
when $k$ is odd and
\begin{align*}
LEN(P^{(3)})-LEN(P_2)&=\frac{i_1+j_0-rn-2j}{k}-(k+1-i_1-j_0) \\
&\ge\frac{i_1+j_0-rn-n}{k}-\frac{k}{2}\\
&>\frac{(-r-1)k^2/2}{k}-\frac{k}{2}\\
&=\frac{(-r-2)k}{2}\ge 0
\end{align*}
when $k$ is even.
So $P^{(3)}$ is a $u_0, v_j$-path but not of minimum length.

(3) $-k<i_1\le 0$ and $r\ge 1$.

In this case, the $u_0,u_{i_1}$-path via outer edges is of negative direction and the $v_{i_1},v_{J}$-path
via inner edges is of positive direction. Note that\\
$LEN(P^{(3)})=|i_1|+1+|\frac{j+rn-i_1}{k}|=-i_1+1+\frac{j+rn-i_1}{k}$.

(3.1) When $k$ is odd and $j_0\le\frac{k+1}{2}$, or $k$ is even and $j_0\le\frac{k}{2}$.

We compare $LEN(P^{(3)})$ with $LEN(P_1)$.

If $-i_1\ge j_0$, $LEN(P^{(3)})-LEN(P_1)=(-i_1+1+\frac{j+rn-i_1}{k})-(j_0+1+\frac{j-j_0}{k})=\frac{rn-i_1+j_0}{k}+(-i_1-j_0)>0$.

If $-i_1<j_0$, $1\le j_0+i_1\le\frac{k+1}{2}$ when $k$ is odd (or $1\le j_0+i_0\le\frac{k}{2}$ when $k$ is even).
Recall that $n\ge \frac{k(k+1)}{2}$ when $k$ is odd (or $n\ge \frac{k^2}{2}$ when $k$ is even). Then
\begin{align*}
LEN(P^{(3)})-LEN(P_1)&=\frac{rn+j_0-i_1}{k}-(j_0+i_1) \\
&>\frac{rk(k+1)/2}{k}-\frac{k+1}{2}\\
&=\frac{(r-1)(k+1)}{2}\ge 0
\end{align*}
when $k$ is odd and
\begin{align*}
LEN(P^{(3)})-LEN(P_1)&=\frac{rn+j_0-i_1}{k}-(j_0+i_1)\\
&>\frac{rk^2/2}{k}-\frac{k}{2}\\
&=\frac{(r-1)k}{2}\ge 0
\end{align*}
when $k$ is even.
So $P^{(3)}$ is a $u_0, v_j$-path but not of minimum length.

(3.2) When $k$ is odd and $j_0>\frac{k+1}{2}$, or $k$ is even and $j_0>\frac{k}{2}$.

We compare $LEN(P^{(3)})$ with $LEN(P_2)$.

If $-i_1\ge j_0$, $-i_1>k-j_0$ and so $-i_1+j_0-k>0$. Then
$LEN(P^{(3)})-LEN(P_2)=(-i_1+1+\frac{j+rn-i_1}{k})-(k-j_0+2+\frac{j-j_0}{k})=\frac{rn-i_1+j_0}{k}+(-i_1+j_0-k-1)>0$.

If $-i_1<j_0$, $j_0+i_1\ge 1$.
Recall that $n\ge \frac{k(k+1)}{2}$ when $k$ is odd (or $n\ge \frac{k^2}{2}$ when $k$ is even). Then
\begin{align*}
LEN(P^{(3)})-LEN(P_2)&=\frac{rn+j_0-i_1}{k}-(k+1+i_1-j_0) \\
&>\frac{rk(k+1)/2}{k}-\frac{k+1}{2}\\
&=\frac{(r-1)(k+1)}{2}\ge 0
\end{align*}
when $k$ is odd and
\begin{align*}
LEN(P^{(3)})-LEN(P_2)&=\frac{rn+j_0-i_1}{k}-(k+1+i_1-j_0)\\
&>\frac{rk^2/2}{k}-\frac{k}{2}\\
&=\frac{(r-1)k}{2}\ge 0
\end{align*}
when $k$ is even.
So $P^{(3)}$ is a $u_0, v_j$-path but not of minimum length.

(4) $-k<i_1\le 0$ and $r\le -2$.

In this case, the $u_0,u_{i_1}$-path via outer edges is of negative direction and the $v_{i_1},v_{J}$-path
via inner edges is of negative direction. Note that\\
$LEN(P^{(3)})=|i_1|+1+|\frac{j+rn-i_1}{k}|=-i_1+1+\frac{i_1-(j+rn)}{k}$.

(4.1) When $k$ is odd and $j_0\le\frac{k+1}{2}$, or $k$ is even and $j_0\le\frac{k}{2}$.

We compare $LEN(P^{(3)})$ with $LEN(P_1)$.

If $-i_1\ge j_0$, $LEN(P^{(3)})-LEN(P_1)=(-i_1+1+\frac{i_1-(j+rn)}{k})-(j_0+1+\frac{j-j_0}{k})=\frac{i_1+j_0-rn-2j}{k}+(-i_1-j_0)>0$.

If $-i_1<j_0$, $1\le j_0+i_1\le\frac{k+1}{2}$ when $k$ is odd (or $1\le j_0+i_0\le\frac{k}{2}$ when $k$ is even).
Recall that $n\ge \frac{k(k+1)}{2}$ when $k$ is odd (or $n\ge \frac{k^2}{2}$ when $k$ is even). Then
\begin{align*}
LEN(P^{(3)})-LEN(P_1)&=\frac{i_1+j_0-rn-2j}{k}-(j_0+i_1) \\
&\ge \frac{i_1+j_0-rn-n}{k}-\frac{k+1}{2}\\
&>\frac{(-r-1)k(k+1)/2}{k}-\frac{k+1}{2}\\
&=\frac{(-r-2)(k+1)}{2}\ge 0
\end{align*}
when $k$ is odd and
\begin{align*}
LEN(P^{(3)})-LEN(P_1)&=\frac{i_1+j_0-rn-2j}{k}-(j_0+i_1) \\
&\ge \frac{i_1+j_0-rn-n}{k}-\frac{k}{2}\\
&>\frac{(-r-1)k^2/2}{k}-\frac{k}{2}\\
&=\frac{(-r-2)k}{2}\ge 0
\end{align*}
when $k$ is even.
So $P^{(3)}$ is a $u_0, v_j$-path but not of minimum length.

(4.2) When $k$ is odd and $j_0>\frac{k+1}{2}$, or $k$ is even and $j_0>\frac{k}{2}$.

We compare $LEN(P^{(3)})$ with $LEN(P_2)$.

If $-i_1\ge j_0$, $-i_1>k-j_0$ and so $-i_1+j_0-k>0$. Then
$LEN(P^{(3)})-LEN(P_2)=(-i_1+1+\frac{i_1-(j+rn)}{k})-(k-j_0+2+\frac{j-j_0}{k})=\frac{i_1+j_0-rn-2j}{k}+(-i_1+j_0-k-1)>0$.

If $-i_1<j_0$, $j_0+i_1\ge 1$.
Recall that $n\ge \frac{k(k+1)}{2}$ when $k$ is odd (or $n\ge \frac{k^2}{2}$ when $k$ is even). Then
\begin{align*}
LEN(P^{(3)})-LEN(P_2)&=\frac{i_1+j_0-rn-2j}{k}-(k+1+i_1-j_0) \\
&>\frac{i_1+j_0-rn-n}{k}-\frac{k+1}{2}\\
&>\frac{(-r-1)k(k+1)/2}{k}-\frac{k+1}{2}\\
&=\frac{(-r-2)(k+1)}{2}\ge 0
\end{align*}
when $k$ is odd and
\begin{align*}
LEN(P^{(3)})-LEN(P_2)&=\frac{i_1+j_0-rn-2j}{k}-(k+1+i_1-j_0) \\
&\ge\frac{i_1+j_0-rn-n}{k}-\frac{k}{2}\\
&>\frac{(-r-1)k^2/2}{k}-\frac{k}{2}\\
&=\frac{(-r-2)k}{2}\ge 0
\end{align*}
when $k$ is even.
So $P^{(3)}$ is a $u_0, v_j$-path but not of minimum length.

The proof of the lemma completes.

\end{proof}

Let $d_{12}(u_0,v_j)$ be the distance between $u_0$ and $v_j$ in $GP(n,k)$ via paths of type $P_1$ or $P_2$.
Let $d_{34}(u_0,v_j)$ be the distance between $u_0$ and $v_j$ in $GP(n,k)$ via paths of type $P_3$ or $P_4$.
Then
$$d(u_0,v_j) = \min\{d_{12}(u_0,v_j),d_{34}(u_0,v_j)\}$$
and
$$d(u_0,v_{j^*}) = \max\{d(u_0,v_j):\ 0\le j\le n/2\}.$$

We first find $j^1$ such that $d_{12}(u_0,v_{j^1}) = \max\{d_{12}(u_0,v_j):\ 0\le j\le n/2\}$.
If $d_{34}(u_0,v_{j^1})\ge d_{12}(u_0,v_{j^1})$, then $j^*=j^1$.
If $d_{34}(u_0,v_{j^1})< d_{12}(u_0,v_{j^1})$, we can find $j^*$ around $j^1$ such that
$|d_{12}(u_0,v_{j^*})-d_{34}(u_0,v_{j^*})|\le 1$ and $\min\{d_{12}(u_0,v_{j^*}),d_{34}(u_0,v_{j^*})\}$ is as large as possible.
Note that $j^*$ is not unique.

The following discussions are organized using a tree with depth $3$.
At depth $1$, the discussions are according to the parity of $k$.
At depth $2$, the discussions are according to the parity of $n$.
At depth $3$, the discussions are according to parameters contained in the small cases.

\medskip\noindent
{\bf Case 1}: $k$ is odd.\\
Notice that $n\ge 3k-1$ in this case. We will prove that there exist a $j^*$ such that
$d(u_0,v_{j^*}) = \max\{d(u_0,v_j):\ 0\le j\le n/2\}$ and $k<j^*\le n/2$.

\medskip\noindent
{\bf Case 1.1}: $n$ is even.\\
Suppose $n/2=m_2k+j_2$ where $0\le j_2<k$.
From $n\ge 3k-1$, we know that $m_2\ge 2$, or $m_2=1$ and $j_2\ge \frac{k-1}{2}$.

\medskip\noindent
{\bf Case 1.1.1}: $j_2\ge \frac{k-1}{2}$.

If $0\le j\le \frac{k+1}{2}$, then $d_{12}(u_0,v_{m_2k+j})=m_2+1+j$.
If $\frac{k+1}{2}< j\le j_2$, then $d_{12}(u_0,v_{m_2k+j})=m_2+k+2-j$.
Observe that $d_{12}(u_0,v_{(m_2-1)k+\frac{k+1}{2}})=m_2+\frac{k+1}{2}$.
Because $j_2\ge \frac{k-1}{2}$, we infer that $m_2+1+j\ge m_2+\frac{k+1}{2}$ when $j=\frac{k-1}{2}$ or $j=\frac{k+1}{2}$.
So we just need to consider the distance between $u_0$ and $v_{m_2k+j}$, where $0\le j\le j_2$.
Note that $d_{12}(u_0,v_{m_2k+\frac{k+1}{2}})=m_2+\frac{k+1}{2}+1$,
$d_{12}(u_0,v_{m_2k+\frac{k+1}{2}-1})=d_{12}(u_0,v_{m_2k+\frac{k+1}{2}+1})=m_2+\frac{k+1}{2}$, and so on.

Note that $v_{-(m_2+1)k}=v_{n-(m_2+1)k}=v_{m_2k+2j_2-k}$. Observe that
$d_{34}(u_0,v_{m_2k+2j_2-k})=m_2+2$,
$d_{34}(u_0,v_{m_2k+2j_2-k+1})=d_{34}(u_0,v_{m_2k+2j_2-k-1})=m_2+3$, and so on.

If $j_2=\frac{k-1}{2}$, then $j^1=m_2k+\frac{k-1}{2}$. Because $d_{34}(u_0,v_{j^1})\ge d_{12}(u_0,v_{j^1})$, we get $j^*=j^1=m_2k+\frac{k-1}{2}$.

If $j_2=\frac{k+1}{2}$, then $j^1=m_2k+\frac{k+1}{2}$. Because $d_{34}(u_0,v_{j^1})\ge d_{12}(u_0,v_{j^1})$, we get $j^*=j^1=m_2k+\frac{k+1}{2}$.

If $3\le 2j_2-k\le \frac{k+1}{2}$, then $j^1=m_2k+\frac{k+1}{2}$ and $d_{34}(u_0,v_{j^1})<d_{12}(u_0,v_{j^1})$. We get $j^*=m_2k+j_2$.

If $2j_2-k> \frac{k+1}{2}$, then $j^1=m_2k+\frac{k+1}{2}$ and $d_{34}(u_0,v_{j^1})<d_{12}(u_0,v_{j^1})$.
We get $j^*=m_2k+j_2-\frac{k-1}{2}$.

Table~\ref{T:case1} shows that how to find $j^*$ in $GP(50,9)$.

\begin{table}
\begin{center}
\begin{tabular}{|c|c|c|c|c|c|c|c|}
\hline
$v_j$ & $v_{20}$ & $v_{21}$  & $v_{22}$ &$v_{23}$ &$v_{24}$ &$v_{25}$     \\
\hline
$d_{12}(u_0,v_j)$ & 5 & 6 & 7 & 8 & 7 & 6   \\
\hline
$d_{34}(u_0,v_j)$ & 7 & 6 & 5 & 4 & 5 & 6   \\
\hline
\end{tabular}
\end{center}
\centerline{}  \caption{\small In $GP(50,9)$, the search of $j^*$. Because $j^1=23$, $d_{12}(u_0,v_{23})=8$,
$n-(m_2+1)k=23$ and $d_{34}(u_0,v_{23})=4$, so $j^*=21$ or $j^*=25$. }
\label{T:case1}
\end{table}

\medskip\noindent
{\bf Case 1.1.2}: $j_2<\frac{k-1}{2}$.

If $0\le j\le j_2$, then $d_{12}(u_0,v_{m_2k+j})=m_2+1+j$.
Note that $d_{12}(u_0,v_{(m_2-1)k+\frac{k+1}{2}})=m_2+\frac{k+1}{2}$.
Because $j_2<\frac{k-1}{2}$, we have $m_2+1+j<m_2+\frac{k+1}{2}$ when $0\le j\le j_2$.
So we just need to consider the distance between $u_0$ and $v_{(m_2-1)k+j}$, where $0\le j\le k$.
Note that $d_{12}(u_0,v_{(m_2-1)k+\frac{k+1}{2}})=m_2+\frac{k+1}{2}$,
$d_{12}(u_0,v_{(m_2-1)k+\frac{k+1}{2}-1})=d_{12}(u_0,v_{(m_2-1)k+\frac{k+1}{2}+1})=m_2+\frac{k+1}{2}-1$,  and so on.


Note that $v_{-(m_2+1)k}=v_{n-(m_2+1)k}=v_{(m_2-1)k+2j_2}$ and $2j_2< k-1$.
Moreover, $d_{34}(u_0,v_{(m_2-1)k+2j_2})=m_2+2$,
$d_{34}(u_0,v_{(m_2-1)k+2j_2+1})=d_{34}(u_0,v_{(m_2-1)k+2j_2-1})=m_2+3$, and so on.

If $j_2=0$ or $j_2=1$, we set $j^1=(m_2-1)k+\frac{k+1}{2}$.
Because $d_{34}(u_0,v_{j^1})\ge d_{12}(u_0,v_{j^1})$, we have $j^*=j^1=(m_2-1)k+\frac{k+1}{2}$.

If $4\le 2j_2\le\frac{k+1}{2}$, we have $j^1=(m_2-1)k+\frac{k+1}{2}$ and $d_{34}(u_0,v_{j^1})<d_{12}(u_0,v_{j^1})$.
We get $j^*=(m_2-1)k+\frac{k+1}{2}+j_2-1$.

If $2j_2>\frac{k+1}{2}$, then $j^1=(m_2-1)k+\frac{k+1}{2}$ and $d_{34}(u_0,v_{j^1})<d_{12}(u_0,v_{j^1})$.
We get $j^*=(m_2-1)k+j_2+1$.

Table~\ref{T:case2} shows that how to find $j^*$ in $GP(42,9)$.

\begin{table}
\begin{center}
\begin{tabular}{|c|c|c|c|c|c|c|c|}
\hline
$v_j$ & $v_{12}$ & $v_{13}$  & $v_{14}$ &$v_{15}$ &$v_{16}$ &$v_{17}$     \\
\hline
$d_{12}(u_0,v_j)$ & 5 & 6 & 7 & 6 & 5 & 4   \\
\hline
$d_{34}(u_0,v_j)$ & 7 & 6 & 5 & 4 & 5 & 6   \\
\hline
\end{tabular}
\end{center}
\centerline{}  \caption{\small In $GP(42,9)$, the search of $j^*$. Because $j^1=14$, $d_{12}(u_0,v_{14})=7$,
$n-(m_2+1)k=15$ and $d_{34}(u_0,v_{15})=4$, so $j^*=13$. }
\label{T:case2}
\end{table}

\medskip\noindent
{\bf Case 1.2}: $n$ is odd.\\
Suppose $(n-1)/2=m_3k+j_3$ where $0\le j_3<k$.
From $n\ge 3k-1$, we know that $m_3\ge 2$, or $m_3=1$ and $j_3\ge \frac{k-2}{2}$.

\medskip\noindent
{\bf Case 1.2.1}: $j_3\ge \frac{k-2}{2}$.

Because $k$ is odd, $\frac{k-2}{2}$ is not an integer and hence $j_3\ge\frac{k-1}{2}$.
It suffices to consider $d(u_0, v_{m_3k+j})$, where $0\le j\le j_3$.
Note that $d_{12}(u_0,v_{m_3k+\frac{k+1}{2}})=m_3+\frac{k+1}{2}+1$,
$d_{12}(u_0,v_{m_3k+\frac{k+1}{2}-1})=d_{12}(u_0,v_{m_3k+\frac{k+1}{2}+1})=m_3+\frac{k+1}{2}$, and so on.

Note that $v_{-(m_3+1)k}=v_{n-(m_3+1)k}=v_{m_3k+2j_3+1-k}$. Moreover, $d_{34}(u_0,v_{m_3k+2j_3+1-k})=m_3+2$,
$d_{34}(u_0,v_{m_3k+2j_3+1-k+1})=d_{34}(u_0,v_{m_3k+2j_3+1-k-1})=m_3+3$,
$d_{34}(u_0,v_{m_3k+2j_3+1-k+2})=d_{34}(u_0,v_{m_3k+2j_3+1-k-2})=m_3+4$, and so on.

If $j_3=\frac{k-1}{2}$, then select $j^1=m_3k+\frac{k-1}{2}$.
Because $d_{34}(u_0,v_{j^1})\ge d_{12}(u_0,v_{j^1})$, we have $j^*=j^1=m_3k+\frac{k-1}{2}$.

If $2\le 2j_3+1-k\le\frac{k+1}{2}$, then $j^1=m_3k+\frac{k+1}{2}$ and $d_{34}(u_0,v_{j^1})<d_{12}(u_0,v_{j^1})$.
We get $j^*=m_3k+j_3$.

If $2j_3+1-k>\frac{k+1}{2}$, then $j^1=m_3k+\frac{k+1}{2}$ and $d_{34}(u_0,v_{j^1})<d_{12}(u_0,v_{j^1})$.
We get $j^*=m_3k+j_3-\frac{k-3}{2}$ or $j^*=m_3k+j_3-\frac{k-1}{2}$.

Table~\ref{T:case3} shows that how to find $j^*$ in $GP(51,9)$.

\begin{table}
\begin{center}
\begin{tabular}{|c|c|c|c|c|c|c|c|}
\hline
$v_j$ & $v_{20}$ & $v_{21}$  & $v_{22}$ &$v_{23}$ &$v_{24}$ &$v_{25}$     \\
\hline
$d_{12}(u_0,v_j)$ & 5 & 6 & 7 & 8 & 7 & 6   \\
\hline
$d_{34}(u_0,v_j)$ & 8 & 7 & 6 & 5 & 4 & 5   \\
\hline
\end{tabular}
\end{center}
\centerline{}  \caption{\small In $GP(51,9)$, the search of $j^*$. Because $j^1=23$, $d_{12}(u_0,v_{23})=8$,
$n-(m_3+1)k=24$ and $d_{34}(u_0,v_{24})=4$, so $j^*=21$ or $j^*=22$. }
\label{T:case3}
\end{table}

\medskip\noindent
{\bf Case 1.2.2}: $j_3< \frac{k-2}{2}$.

It suffices to consider the distances $d(u_0, v_{(m_3-1)k+j})$, where $0\le j\le k$. We infer that $d_{12}(u_0,v_{(m_3-1)k+\frac{k+1}{2}})=m_3+\frac{k+1}{2}$,
$d_{12}(u_0,v_{(m_3-1)k+\frac{k+1}{2}-1})=d_{12}(u_0,v_{(m_3-1)k+\frac{k+1}{2}+1})=m_3+\frac{k+1}{2}-1$, and so on.

Note that $v_{-(m_3+1)k}=v_{n-(m_3+1)k}=v_{(m_3-1)k+2j_3+1}$ and $2j_3+1< k-1$.
Moreover,  $d_{34}(u_0,v_{(m_3-1)k+2j_3+1})=m_3+2$,
$d_{34}(u_0,v_{(m_3-1)k+2j_3+1+1})=d_{34}(u_0,v_{(m_3-1)k+2j_3+1-1})=m_3+3$, and so on.

If $j_3=0$, then let $j^1=(m_3-1)k+\frac{k+1}{2}$.
Because $d_{34}(u_0,v_{j^1})\ge d_{12}(u_0,v_{j^1})$, we have $j^*=j^1=(m_3-1)k+\frac{k+1}{2}$.

If $3\le 2j_3+1\le\frac{k+1}{2}$, then $j^1=(m_3-1)k+\frac{k+1}{2}$ and $d_{34}(u_0,v_{j^1})<d_{12}(u_0,v_{j^1})$.
We get $j^*=(m_3-1)k+\frac{k+1}{2}+j_3-1$ or $j^*=(m_3-1)k+\frac{k+1}{2}+j_3$.

If $2j_3+1>\frac{k+1}{2}$, then $j^1=(m_3-1)k+\frac{k+1}{2}$ and $d_{34}(u_0,v_{j^1})<d_{12}(u_0,v_{j^1})$.
We get $j^*=(m_3-1)k+j_3+2$ or $j^*=(m_3-1)k+j_3+1$.

Table~\ref{T:case4} shows that how to find $j^*$ in $GP(43,9)$.

\begin{table}
\begin{center}
\begin{tabular}{|c|c|c|c|c|c|c|c|}
\hline
$v_j$ & $v_{12}$ & $v_{13}$  & $v_{14}$ &$v_{15}$ &$v_{16}$ &$v_{17}$     \\
\hline
$d_{12}(u_0,v_j)$ & 5 & 6 & 7 & 6 & 5 & 4   \\
\hline
$d_{34}(u_0,v_j)$ & 8 & 7 & 6 & 5 & 4 & 5   \\
\hline
\end{tabular}
\end{center}
\centerline{}  \caption{\small In $GP(43,9)$, the search of $j^*$. Because $j^1=14$, $d_{12}(u_0,v_{14})=7$,
$n-(m_3+1)k=16$ and $d_{34}(u_0,v_{16})=4$, so $j^*=13$ or $j^*=14$. }
\label{T:case4}
\end{table}

\medskip\noindent
{\bf Case 2}: $k$ is even.\\
Notice that $n\ge 3k-2$ in this case. We will prove that there exists $j^*$ such that
$d(u_0,v_{j^*}) = \max\{d(u_0,v_j):\ 0\le j\le n/2\}$ and $k<j^*\le n/2$.

\medskip\noindent
{\bf Case 2.1}: $n$ is even.\\
Suppose $n/2=m_4k+j_4$ where $0\le j_4<k$.
From $n\ge 3k-2$, we know that $m_4\ge 2$, or $m_4=1$ and $j_4\ge \frac{k-2}{2}$.

\medskip\noindent
{\bf Case 2.1.1}: $j_4\ge \frac{k-2}{2}$.

It suffices to consider the distances $d(u_0, v_{m_4k+j})$, where $0\le j\le j_4$.
Note that $d_{12}(u_0,v_{m_4k+\frac{k}{2}})=d_{12}(u_0,v_{m_4k+\frac{k+2}{2}})=m_4+\frac{k}{2}+1$,
$d_{12}(u_0,v_{m_4k+\frac{k}{2}-1})=d_{12}(u_0,v_{m_4k+\frac{k+2}{2}+1})=m_4+\frac{k}{2}$, and so on.

Observe that $v_{-(m_4+1)k}=v_{n-(m_4+1)k}=v_{m_4k+2j_4-k}$. Moreover,
$d_{34}(u_0,v_{m_4k+2j_4-k})=m_4+2$,
$d_{34}(u_0,v_{m_4k+2j_4-k+1})=d_{34}(u_0,v_{m_4k+2j_4-k-1})=m_4+3$,
$d_{34}(u_0,v_{m_4k+2j_4-k+2})=d_{34}(u_0,v_{m_4k+2j_4-k-2})=m_4+4$, and so on.

If $j_4=\frac{k-2}{2}$, then $j^1=m_4k+\frac{k-2}{2}$.
Because $d_{34}(u_0,v_{j^1})\ge d_{12}(u_0,v_{j^1})$, we have $j^*=j^1=m_4k+\frac{k-2}{2}$.

If $j_4=\frac{k}{2}$, then $j^1=m_4k+\frac{k}{2}$.
Because $d_{34}(u_0,v_{j^1})\ge d_{12}(u_0,v_{j^1})$, we have $j^*=j^1=m_4k+\frac{k}{2}$.

If $2\le 2j_4-k\le\frac{k}{2}$, then $j^1=m_4k+\frac{k}{2}$ and $d_{34}(u_0,v_{j^1})<d_{12}(u_0,v_{j^1})$.
We get $j^*=m_4k+j_4$.

If $2j_4-k\ge\frac{k+2}{2}$, then $j^1=m_4k+\frac{k}{2}$ and $d_{34}(u_0,v_{j^1})<d_{12}(u_0,v_{j^1})$.
We get $j^*=m_4k+j_4-\frac{k}{2}+1$ or $j^*=m_4k+j_4-\frac{k}{2}$.

Table~\ref{T:case5} shows that how to find $j^*$ in $GP(56,10)$.

\begin{table}
\begin{center}
\begin{tabular}{|c|c|c|c|c|c|c|c|}
\hline
$v_j$ & $v_{22}$ & $v_{23}$  & $v_{24}$ &$v_{25}$ &$v_{26}$ &$v_{27}$ &$v_{28}$   \\
\hline
$d_{12}(u_0,v_j)$ & 5 & 6 & 7 & 8 & 8 & 7 & 6 \\
\hline
$d_{34}(u_0,v_j)$ & 8 & 7 & 6 & 5 & 4 & 5 & 6 \\
\hline
\end{tabular}
\end{center}
\centerline{}  \caption{\small In $GP(56,10)$, the search of $j^*$. Because $j^1=25$, $d_{12}(u_0,v_{25})=8$,
$n-(m_4+1)k=26$ and $d_{34}(u_0,v_{26})=4$, so $j^*=23$, $j^*=24$ or $j^*=28$. }
\label{T:case5}
\end{table}

\medskip\noindent
{\bf Case 2.1.2}: $j_4< \frac{k-2}{2}$.

It is enough to consider the distances $d(u_0, v_{(m_4-1)k+j})$, where $0\le j\le k$.
Note that $d_{12}(u_0,v_{(m_4-1)k+\frac{k}{2}})=d_{12}(u_0,v_{(m_4-1)k+\frac{k+2}{2}})=m_4+\frac{k}{2}$,
$d_{12}(u_0,v_{(m_4-1)k+\frac{k}{2}-1})=d_{12}(u_0,v_{(m_4-1)k+\frac{k+2}{2}+1})=m_4+\frac{k}{2}-1$, and so on.

Note that $v_{-(m_4+1)k}=v_{n-(m_4+1)k}=v_{(m_4-1)k+2j_2}$ and $2j_2< k-2$.
We also have $d_{34}(u_0,v_{(m_4-1)k+2j_2})=m_4+2$,
$d_{34}(u_0,v_{(m_4-1)k+2j_2+1})=d_{34}(u_0,v_{(m_4-1)k+2j_2-1})=m_4+3$,
$d_{34}(u_0,v_{(m_4-1)k+2j_2+2})=d_{34}(u_0,v_{(m_4-1)k+2j_2-2})=m_4+4$, and so on.

If $j_4=0$ or $j_4=1$, then $j^1=(m_4-1)k+\frac{k}{2}$.
Because $d_{34}(u_0,v_{j^1})\ge d_{12}(u_0,v_{j^1})$, we have $j^*=j^1=(m_4-1)k+\frac{k}{2}$.

If $4\le 2j_4\le\frac{k}{2}$, then $j^1=(m_4-1)k+\frac{k}{2}$ and $d_{34}(u_0,v_{j^1})<d_{12}(u_0,v_{j^1})$.
We get $j^*=(m_4-1)k+\frac{k}{2}+j_4-1$ or $j^*=(m_4-1)k+\frac{k}{2}+j_4$.

If $2j_4\ge\frac{k+2}{2}$, then $j^1=(m_4-1)k+\frac{k}{2}$ and $d_{34}(u_0,v_{j^1})<d_{12}(u_0,v_{j^1})$.
We get $j^*=(m_4-1)k+j_4+1$.

Table~\ref{T:case6} shows that how to find $j^*$ in $GP(44,10)$.

\begin{table}
\begin{center}
\begin{tabular}{|c|c|c|c|c|c|c|c|}
\hline
$v_j$ & $v_{12}$ & $v_{13}$  & $v_{14}$ &$v_{15}$ &$v_{16}$ &$v_{17}$ &$v_{18}$   \\
\hline
$d_{12}(u_0,v_j)$ & 4 & 5 & 6 & 7 & 7 & 6 & 5 \\
\hline
$d_{34}(u_0,v_j)$ & 6 & 5 & 4 & 5 & 6 & 7 & 8 \\
\hline
\end{tabular}
\end{center}
\centerline{}  \caption{\small In $GP(44,10)$, the search of $j^*$. Because $j^1=15$, $d_{12}(u_0,v_{15})=7$,
$n-(m_4+1)k=14$ and $d_{34}(u_0,v_{14})=4$, so $j^*=16$ or $j^*=17$. }
\label{T:case6}
\end{table}

\medskip\noindent
{\bf Case 2.2}: $n$ is odd. \\
Suppose $(n-1)/2=m_5k+j_5$, where $0\le j_5<k$.
From $n\ge 3k-2$, we know that $m_5\ge 2$, or $m_5=1$ and $j_5\ge \frac{k-3}{2}$.

\medskip\noindent
{\bf Case 2.2.1}: $j_5\ge \frac{k-3}{2}$.

Because $k$ is even, $\frac{k-3}{2}$ is not an integer and hence $j_5\ge\frac{k-2}{2}$. Again it suffices to consider the distances $d(u_0, v_{m_5k+j})$, where $0\le j\le j_5$.
Note that $d_{12}(u_0,v_{m_5k+\frac{k}{2}})=d_{12}(u_0,v_{m_5k+\frac{k+2}{2}})=m_5+\frac{k}{2}+1$,
$d_{12}(u_0,v_{m_5k+\frac{k}{2}-1})=d_{12}(u_0,v_{m_5k+\frac{k+2}{2}+1})=m_5+\frac{k}{2}$, and so on.

Observe that $v_{-(m_5+1)k}=v_{n-(m_5+1)k}=v_{m_5k+2j_5+1-k}$. Also,  $d_{34}(u_0,v_{m_5k+2j_5+1-k})=m_5+2$,
$d_{34}(u_0,v_{m_5k+2j_5+1-k+1})=d_{34}(u_0,v_{m_5k+2j_5+1-k-1})=m_5+3$,
$d_{34}(u_0,v_{m_5k+2j_5+1-k+2})=d_{34}(u_0,v_{m_5k+2j_5+1-k-2})=m_5+4$, and so on.

If $j_5=\frac{k-2}{2}$, then $j^1=m_5k+\frac{k-2}{2}$.
Because $d_{34}(u_0,v_{j^1})\ge d_{12}(u_0,v_{j^1})$, we have $j^*=j^1=m_5k+\frac{k-2}{2}$.

If $j_5=\frac{k}{2}$, then $j^1=m_5k+\frac{k}{2}$.
Because $d_{34}(u_0,v_{j^1})\ge d_{12}(u_0,v_{j^1})$, we infer that $j^*=j^1=m_5k+\frac{k}{2}$.

If $3\le 2j_5+1-k\le\frac{k}{2}$, then $j^1=m_5k+\frac{k}{2}$ and $d_{34}(u_0,v_{j^1})<d_{12}(u_0,v_{j^1})$.
We get $j^*=m_5k+j_5$.

If $2j_5+1-k\ge\frac{k+2}{2}$, then $j^1=m_5k+\frac{k}{2}$ and $d_{34}(u_0,v_{j^1})<d_{12}(u_0,v_{j^1})$.
We get $j^*=m_5k+j_5+1-\frac{k}{2}$.

Table~\ref{T:case7} shows that how to find $j^*$ in $GP(57,10)$.

\begin{table}
\begin{center}
\begin{tabular}{|c|c|c|c|c|c|c|c|}
\hline
$v_j$ & $v_{23}$ & $v_{24}$  & $v_{25}$ &$v_{26}$ &$v_{27}$ &$v_{28}$    \\
\hline
$d_{12}(u_0,v_j)$ & 6 & 7 & 8 & 8 & 7 & 6  \\
\hline
$d_{34}(u_0,v_j)$ & 8 & 7 & 6 & 5 & 4 & 5  \\
\hline
\end{tabular}
\end{center}
\centerline{}  \caption{\small In $GP(57,10)$, the search of $j^*$. Because $j^1=25$, $d_{12}(u_0,v_{25})=8$,
$n-(m_5+1)k=27$ and $d_{34}(u_0,v_{27})=4$, so $j^*=24$. }
\label{T:case7}
\end{table}

\medskip\noindent
{\bf Case 2.2.2}: $j_5< \frac{k-3}{2}$.

It suffices to consider the distances $d(u_0, v_{(m_5-1)k+j})$, where $0\le j\le k$.
Note that $d_{12}(u_0,v_{(m_5-1)k+\frac{k}{2}})=d_{12}(u_0,v_{(m_5-1)k+\frac{k+2}{2}})=m_5+\frac{k}{2}$,
$d_{12}(u_0,v_{(m_5-1)k+\frac{k}{2}-1})=d_{12}(u_0,v_{(m_5-1)k+\frac{k+2}{2}+1})=m_5+\frac{k}{2}-1$, and so on.

Note that $v_{-(m_5+1)k}=v_{n-(m_5+1)k}=v_{(m_5-1)k+2j_5+1}$ and $2j_5+1< k-2$. We have
$d_{34}(u_0,v_{(m_5-1)k+2j_5+1})=m_5+2$,
$d_{34}(u_0,v_{(m_5-1)k+2j_5+1+1})=d_{34}(u_0,v_{(m_5-1)k+2j_5+1-1})=m_5+3$,
$d_{34}(u_0,v_{(m_5-1)k+2j_5+1+2})=d_{34}(u_0,v_{(m_5-1)k+2j_5+1-2})=m_5+4$, and so on.

If $j_5=0$, then $j^1=(m_5-1)k+\frac{k}{2}$.
Because $d_{34}(u_0,v_{j^1})\ge d_{12}(u_0,v_{j^1})$, we infer that $j^*=j^1=(m_5-1)k+\frac{k}{2}$.

If $3\le 2j_5+1\le\frac{k}{2}$, then $j^1=(m_5-1)k+\frac{k}{2}$ and $d_{34}(u_0,v_{j^1})<d_{12}(u_0,v_{j^1})$.
We get $j^*=(m_5-1)k+\frac{k}{2}+j_5$.

If $2j_5+1\ge\frac{k+2}{2}$, then $j^1=(m_5-1)k+\frac{k}{2}$ and $d_{34}(u_0,v_{j^1})<d_{12}(u_0,v_{j^1})$.
We get $j^*=(m_5-1)k+j_5+2$ or $j^*=(m_5-1)k+j_5+1$.

Table~\ref{T:case8} shows that how to find $j^*$ in $GP(45,10)$.

\begin{table}
\begin{center}
\begin{tabular}{|c|c|c|c|c|c|c|c|}
\hline
$v_j$ & $v_{12}$ & $v_{13}$  & $v_{14}$ &$v_{15}$ &$v_{16}$ &$v_{17}$ &$v_{18}$   \\
\hline
$d_{12}(u_0,v_j)$ & 4 & 5 & 6 & 7 & 7 & 6 & 5 \\
\hline
$d_{34}(u_0,v_j)$ & 7 & 6 & 5 & 4 & 5 & 6 & 7 \\
\hline
\end{tabular}
\end{center}
\centerline{}  \caption{\small In $GP(45,10)$, the search of $j^*$. Because $j^1=15$, $d_{12}(u_0,v_{15})=7$,
$n-(m_5+1)k=15$ and $d_{34}(u_0,v_{15})=4$, so $j^*=17$. }
\label{T:case8}
\end{table}

Suppose that $P^*$ is a shortest $u_0,v_{j^*}$-path.
Let $P^*+v_{j^*}u_{j^*}$ be the path obtained from $P^*$ by appending the edge $v_{j^*}u_{j^*}$ at $v_{j^*}$.
Because $k\ge3$ and $k<j^*\le n/2$, the path $P^*+v_{j^*}u_{j^*}$ is a shortest $u_0,u_{j^*}$-path
which contains $v_{j^*}$. We conclude that $d(u_0,v_{j^*})<d(u_0,u_{j^*})\le D$.

\section{On the diameter of $GP(n,k)$}
\label{S:GP-diameter}

In the previous section we found a $j^*$, where $0\le j^*\le n/2$, such that $d(u_0,v_{j^*}) = \max\{d(u_0,v_j):\ 0\le j< n\}$. In fact, the proof also reveals that
$$\diam(GP(n,k)) = d(u_0,u_{j^*}) = d(u_0,v_{j^*})+1\,,$$
which in turn enables us to state the following theorem.

\begin{theorem}\label{T:GP-diameter}
If $n$ and $k$ are integers, where $3\le k< n/2$ and
\begin{equation*}
n\ge\left\{\begin{array}{ll}
8; & k=3, \\
10; & k=4, \\
\frac{k(k+1)}{2}; & k \ \text{is odd and}\ k\ge 5,\\
\frac{k^2}{2}; & k \ \text{is even and}\ k\ge 6,
\end{array}\right.
\end{equation*}
then the following hold.
\begin{enumerate}
\item If $k\ge 3$, $k$ is odd, $n$ is even, and $\frac{n}{2}=mk+j$, where $\frac{k-1}{2}\le j<k$, then
\begin{equation*}
\diam(GP(n,k)) = \left\{\begin{array}{ll}
m+2+j; & j = \frac{k-1}{2}\ \text{or}\ j=\frac{k+1}{2}, \\
m+3+k-j; & 3\le 2j-k\le\frac{k+1}{2}, \\
m+2+j-\frac{k-1}{2}; & 2j-k>\frac{k+1}{2}.
\end{array}\right.
\end{equation*}
\item If $k\ge 3$, $k$ is odd, $n$ is even, and $\frac{n}{2}=mk+j$, where $0\le j<\frac{k-1}{2}$, then
\begin{equation*}
\diam(GP(n,k)) = \left\{\begin{array}{ll}
m+1+\frac{k+1}{2}; & j=0 \ \text{or}\ j=1, \\
m+3+\frac{k-1}{2}-j; & 4\le 2j\le\frac{k+1}{2}, \\
m+2+j; & 2j>\frac{k+1}{2}.
\end{array}\right.
\end{equation*}
\item If $k\ge 3$, $k$ is odd, $n$ is odd, and $\frac{n-1}{2}=mk+j$, where $\frac{k-2}{2}\le j<k$, then
\begin{equation*}
\diam(GP(n,k)) = \left\{\begin{array}{ll}
m+2+\frac{k-1}{2}; & j=\frac{k-1}{2}, \\
m+2+k-j; & 2\le 2j+1-k\le\frac{k+1}{2}, \\
m+2+j-\frac{k-1}{2}; & 2j+1-k>\frac{k+1}{2}.
\end{array}\right.
\end{equation*}
\item If $k\ge 3$, $k$ is odd, $n$ is odd, and $\frac{n-1}{2}=mk+j$, where $0\le j<\frac{k-2}{2}$, then
\begin{equation*}
\diam(GP(n,k)) = \left\{\begin{array}{ll}
m+2+\frac{k-1}{2}-j; & 1\le 2j+1\le\frac{k+1}{2}, \\
m+2+j; & 2j+1>\frac{k+1}{2}.
\end{array}\right.
\end{equation*}
\item If $k\ge 4$, $k$ is even, $n$ is even, and $\frac{n}{2}=mk+j$, where $\frac{k-2}{2}\le j<k$, then
\begin{equation*}
\diam(GP(n,k)) = \left\{\begin{array}{ll}
m+2+j; & j=\frac{k-2}{2}\ \text{or}\ j=\frac{k}{2}, \\
m+3+k-j; & 2\le 2j-k\le\frac{k}{2}, \\
m+2+j-\frac{k}{2}; & 2j-k\ge\frac{k+2}{2}.
\end{array}\right.
\end{equation*}
\item If $k\ge 4$, $k$ is even, $n$ is even, and $\frac{n}{2}=mk+j$, where $0\le j<\frac{k-2}{2}$, then
\begin{equation*}
\diam(GP(n,k)) = \left\{\begin{array}{ll}
m+1+\frac{k}{2}; & j=0, \\
m+2+\frac{k}{2}-j; & 2\le 2j\le\frac{k}{2}, \\
m+2+j; & 2j\ge\frac{k+2}{2}.
\end{array}\right.
\end{equation*}
\item If $k\ge 4$, $k$ is even, $n$ is odd, and $\frac{n-1}{2}=mk+j$, where $\frac{k-3}{2}\le j<k$, then
\begin{equation*}
\diam(GP(n,k)) = \left\{\begin{array}{ll}
m+2+\frac{k-2}{2}; & j=\frac{k-2}{2}, \\
m+2+k-j; & 1\le 2j+1-k\le\frac{k}{2}, \\
m+3+j-\frac{k}{2}; & 2j+1-k\ge\frac{k+2}{2}.
\end{array}\right.
\end{equation*}
\item If $k\ge 4$, $k$ is even, $n$ is odd, and $\frac{n-1}{2}=mk+j$, where $0\le j<\frac{k-3}{2}$, then
\begin{equation*}
\diam(GP(n,k)) = \left\{\begin{array}{ll}
m+1+\frac{k}{2}; & j=0, \\
m+2+\frac{k}{2}-j; & 3\le 2j+1\le\frac{k}{2}, \\
m+2+j; & 2j+1\ge\frac{k+2}{2}.
\end{array}\right.
\end{equation*}
\end{enumerate}
\end{theorem}

\section{Concluding remarks}
\label{S:conluding}

In this paper we proved that $GP(n,k)$ is $\diam(GP(n,k))$-distance-balanced provided that $n$ is large enough relative to $k$. In these cases we also determined $\diam(GP(n,k))$. For small values of $k$, we can strengthen these results as follows.

From~\cite{Miklavic:2018} we know that $GP(n,2)$, $n\ge 5$, is $diam(GP(n,2))$-distance-balanced. For $k=2$ and $n\ge 5$, $\diam(GP(n,2))$ can also be computed. First, $\diam(GP(5,2))=2$, $\diam(GP(6,2))=4$, and $\diam(GP(7,2))=3$. Moreover, if $n=4m$ or $n=4m+1$, then $\diam(GP(n,2))=m+2$, and if $n=4m+2$ or $n=4m+3$, then $\diam(GP(n,2))=m+3$.

It is straightforward to check that $\diam(GP(7,3))=3$ and that $GP(7,3)$ is highly distance-balanced. Similarly,  $\diam(GP(9,4))=4$ and $GP(9,4)$ is $4$-distance-balanced. In addition, from~\cite{Miklavic:2018} we recall that $\diam(GP(11,5)) = \diam(GP(14,5)) = 5$, $\diam(GP(12,5)) = \diam(GP(13,5)) = 4$, and that $GP(n,5)$ is $\diam(GP(n,5))$-distance-balanced for $11\le n\le 14$. Moreover, $\diam(GP(n,6))=5$ and $GP(n,6)$ is $5$-distance-balanced for $13\le n\le 17$.

Combining the above results with Theorems~\ref{T:GP-DDB} and~\ref{T:GP-diameter}, the following result can be stated.

\begin{proposition}
If $k$ and $n$ are integers, where $2\le k\le 6$ and $n\ge 2k+1$, then $GP(n,k)$ is $\diam(GP(n,k))$-distance-balanced. Moreover, $\diam(GP(n,k))$ can be computed.
\end{proposition}

For $k\ge 7$ the remaining cases to be solved are collected as follows.

\begin{problem}
Let $k$ and $n$ be two integers, where $k\ge 7$.  Moreover, if $k$ is odd, then $2k+1\le n<\frac{k(k+1)}{2}$ and if $k$ is even, then $k\ge 8$ and $2k+1\le n<\frac{k^2}{2}$.
\begin{enumerate}
\item Is $GP(n,k)$ $\diam(GP(n,k))$-distance balanced?
\item Compute $\diam(GP(n,k))$.
\end{enumerate}
\end{problem}

Moreover, the $\ell$-distance-balancedness of $GP(n,k)$, where $\ell < \diam(GP(n,k))$, is widely open.

\begin{problem}
Let $n$ and $k$ be integers, where $n\ge 5$ and $2\le k< n/2$. For $1\le\ell<\diam(GP(n,k))$ determine whether $GP(n,k)$ is $\ell$-distance-balanced or not.
\end{problem}

\section*{Acknowledgments}

This work was supported by Shandong Provincial Natural Science Foundation of China (ZR2022MA077), the research grant NSFC (11971274) of China and IC Program of Shandong Institutions of Higher Learning For Youth Innovative Talents. Sand Klav\v{z}ar was supported by the Slovenian Research Agency (ARRS) under the grants P1-0297, J1-2452, N1-0285.

\section*{Conflict of interest statement}

On behalf of all authors, the corresponding author states that there is no conflict of interest.

\section*{Data availability statement}

Data sharing not applicable to this article as no datasets were generated or analysed during the current study.

\end{document}